\documentclass[a4paper,10pt]{article}
\usepackage[english]{babel}
\usepackage{hyperref} 
\usepackage[utf8]{inputenc}   
\usepackage{yhmath}  
\usepackage{fixltx2e} 
\usepackage{cancel} 
\usepackage{epsfig}
\usepackage{amsmath}
\usepackage{amsthm}
\usepackage{color}
\usepackage{amssymb}
\usepackage{latexsym}
\usepackage{amsfonts}
\usepackage{txfonts}
\usepackage{pxfonts}
\usepackage{wasysym}
\usepackage{amsmath}
\usepackage[mathcal]{euscript}

\newtheorem{theorem}{Theorem}[section]
\newtheorem{proposition}[theorem]{Proposition}

\theoremstyle{definition}
\newtheorem{definition}[theorem]{Definition}

\title{Rank and border rank of real ternary cubics
}
\author{Maurizio Banchi}

\begin{document}
\maketitle
\begin{abstract}
In this work we give the classification of real ternary cubic forms with respect to rank and border rank up to $SL(3)$-action and
we examine the differences with the complex case.

\end{abstract}
\section{Introduction}
In this paper we solve the classification's problem of real ternary cubic forms with respect to rank and border rank.
This problem is related to the representation of symmetric tensors with real coefficients and it is 
relevant in applications as Electrical Engineering (such as Antenna Array Processing), in Algebraic Statistic, 
in Computer Science, in Data Analysis and in other scientific areas.\\
This topic is a particular case of the more general "Waring decomposition`` of a polynomial
as a sum of powers that goes back to J.J. Sylvester,A. Clebsch,
A.Palatini, A.Terracini and many other mathematicians of XIX and XX centuries.
The decomposition of forms as a sum of linear powers over $\mathbb{R}$ is actually not so developed as its complex 
counterpart because over $\mathbb{R}$ there can be more than one generic rank.\\
Only recently, Comon and Ottaviani (\cite{COM-OTT}) conjectured a list of typical ranks for binary real forms over
the reals and G.Blekherman was able to prove this conjecture (\cite{BLE}).\\
Let $V$ be a vector space of dimension n+1 over the field $\mathbb{R}$, let $S^{d}V$ be the space of symmetric tensors of order d, and let $f\in S^{d}V$.\\
The definition of real rank for a homogeneous polynomial of degree d, $f\in \mathbb{R}[x_{0},..,x_{n}]$ is:\\
\begin{definition}
The (Waring or symmetric) real rank of $f$, denoted by $rk_{\mathbb{R}}(f)$, is the minimum integer r such that
\begin{equation}
 f=\sum_{i=1}^{r}\lambda_{i}l_{i}^{d}
\end{equation}
where $l_i$ are linear forms and $\lambda_{i}\in\mathbb{R}$.
\end{definition}
Since the field is $\mathbb{R}$, the coefficients $\lambda_{i}$ can be reduced to be only $\pm 1$, while in the complex case, we
can impose $\lambda_{i}=1$ for all i.\\
The (symmetric) real border rank of a polynomial $P$, denoted by $\underline{rk}_{\mathbb{R}}(P)$, is defined in terms of limit:
\begin{definition}
 The (symmetric) real border rank of a homogeneous polynomial $P$, denoted by $\underline{rk}_{\mathbb{R}}(P)$, is the smallest
 positive integer $r$ such that there
exists a family of polynomials $P_{\epsilon}$, each of real rank $r$, such that 
\begin{equation}
 P=\lim_{\epsilon \rightarrow 0}P_{\epsilon}.
\end{equation}
\end{definition}
We may define in the same way $rk_{\mathbb{C}}(f)$ and $\underline{rk}_{\mathbb{C}}(f)$ for a homogeneous polynomial $f$ 
(see \cite{LAN} ch. 2).\\
We remark that, for a  general tensor $T$, the rank depends on the field but 
for any field $rk(T) \geq \underline {rk}(T)$.\\
We get for real polynomial $f$, $rk_{\mathbb{R}}(f)\ge rk_{\mathbb{C}}(f)$,
$\underline{rk}_{\mathbb{R}}(f)\ge \underline{r}k_{\mathbb{C}}(f)$
 and inequality can be strict between real and complex rank as the
following example
\begin{displaymath}
 f=2x^3-6xy^2=(x+iy)^{3}+(x-iy)^{3}=(\sqrt[3]{4}x)^3-(x+y)^3-(x-y)^3
\end{displaymath}
shows.\\
In the above example,
$$
\underline{rk}_{\mathbb{C}}(f)=rk_{\mathbb{C}}(f)=2
$$ 
while 
$$
\underline{rk}_{\mathbb{R}}(f)=rk_{\mathbb{R}}(f)=3.
$$
Sylvester gave a method to compute the symmetric rank of a symmetric tensor in $\mathbb{P}(S^dV)$ when dim(V)=2
and Comas and Seiguer implemented this method by giving a complete classification over the complex numbers (\cite{COM-SEI}).\\
Over $\mathbb{R}$ there are algorithms for computing rank of a general real binary form of degree $d=4$ and $d=5$ (see \cite{COM-OTT}, 
\cite{COM-MOU},\cite{BLE} for the decomposition of real forms).\\
The research of an explicit decomposition algorithm  of a (symmetric) tensor is an open problem  and much of the paper 
of Landsberg-Teitler  (\cite{LAN-TEI})is devoted to such important  area.\\
\section{Statement of main result}
The aim of this paper is to compute explicitly the (real) rank for real ternary cubic forms and to show that the maximal real rank
five is obtained in three cases,that is, the union
of a conic and a tangent line (see \cite{LAN-TEI}),the 
new two cases of the cubic that factors as the union of a  conic and an external line and the imaginary triangle (see \S6.3 and \S6.4).\\
We compare our Table 1 with the Table 2 obtained in \cite{LAN-TEI} where the ranks and border ranks 
of plane cubic curves are computed over $\mathbb{C}$.
Over $\mathbb{R}$ we obtained that the difference between the rank and the border rank is 5-3=2 in the case of real conic
plus tangent line  and  5-4=1 in the case of real conic plus external line and the imaginary triangle.
We also find the real decomposition of normal forms in each case under the action of $SL(3)$.
Our result is obtained by looking at the singularity of the Hessian of each normal forms.
The main result is resumed in the following table  that shows the rank and border rank of ternary cubics on $\mathbb{R}$, up to $SL(3)$-action.\\
\begin{theorem}\label{Maurizio-Giorgio}
Ranks and border ranks of real cubics, up to $SL(3)$, are as in Table 1.
\end{theorem}
\begin{table}[htb] \centering
 \begin {tabular}{lllll}
 \textbf{Description}&\textbf{normal form}&\textbf{$rk_{\mathbb{R}}$}&\textbf{$\underline{rk}_{\mathbb{R}}$}
 &\textbf{Hessian}\\ \hline
  1)triple line&$x^3$&1&1& \\ \hline
  2)im. concurrent lines&$x(x^2+y^2)$&2&2& \\ \hline
  3)real concurrent lines&$x(x^2-y^2)$&3&3& \\ \hline
  4)double line+line&$x^2y$&3&2& \\ \hline   
  5)im. conic+line&$(x^2+y^2+z^2)x$&4&4&$-x(-3x^2+y^2+z^2)$\\ 
  & & & & \textbf{real conic+ext. line} \\ \hline
  6)real conic+ext. line&$(x^2+y^2-z^2)z$&5&4&$-z(x^2+y^2+3z^2)$\\ 
  & & & & \textbf{im. conic+line} \\ \hline
  7)real conic+secant line &$(x^2+y^2-z^2)y$&4&4&$y(x^2-3y^2-z^2)$\\ 
  & & & & \textbf{real conic+line} \\ \hline
  8)real conic+tangent line&$18y(x^2+yz)$&5&3&$(y-z)^3$\\ 
  & & & & \textbf{triple line} \\ \hline
  9)real Fermat (Hesse $\lambda =0)$&$x^3+y^3+z^3$&3&3&$xyz$\\
   & & & & \textbf{real triangle} \\ \hline
  10)im. Fermat (Hesse $\lambda =1$)&$x^3+y^3+z^3+6xyz$&4&3&$x^3+y^3+z^3-3xyz$\\ 
  & & & & \textbf{im. triangle} \\ \hline
  11)Hesse pencil $\lambda \neq -\frac{1}{2},0,1$&$x^3+y^3+z^3+6\lambda xyz$&4&4&$-\lambda (x^3+y^3+z^3)$\\
  & & & & $+(1+2\lambda ^3)xyz$ \\ 
  & & & & \textbf{Hesse pencil} \\ \hline
  12)im. triangle (Hesse $\lambda =-\frac{1}{2})$&$x^3+y^3+z^3-3xyz$&5&4&$x^3+y^3+z^3-3xyz$\\ 
  & & & & \textbf{im. triangle} \\ \hline
  13)cusp &$y^2z-x^3$&4&3&$xy^2$\\ 
  & & & & \textbf{double line +line} \\ \hline
  14)nodal cubic &$x^3+y^3+6xyz$&4&4&$xyz+x^3+y^3$\\ 
  & & & & \textbf{nodal cubic} \\ \hline
  15)cubica punctata &$y^2z-x^3+x^2z$ &4&4&$3xy^2-x^2z-y^2z$\\
  & & & & \textbf{cubica punctata} \\ \hline
  16)real triangle &$xyz$&4&4&$xyz$\\
  & & & & \textbf{real triangle}\\ \hline    
 \end{tabular}
 \caption{Ranks and border ranks of ternary cubics on $\mathbb{R}$} \label{tabrank}
\end{table}
\begin{proof}
The proof of theorem \ref{Maurizio-Giorgio} will be divided in the following sections where each case is settled.
\end{proof}
(The values of the border ranks in Table 1 are deduced by the values of Table 2, with an easy computation).
It follows from theorem 2.1 that 4 is the only typical rank for real ternary cubics. This has been proved also in 
\cite{BO}.\\
The following theorem and Table 2 are in \cite{LAN-TEI}, \cite{LAN}, and appeared first in (\cite{Seg}pag.142-143).
We modified the canonical form of orbit 9) of the Table 2) in order to have uniform notation with the real case and
we correct a small misprint that appeared in errata corrige in Landsberg's web page.
\begin{theorem}\label{landsberg-teitler}(L.-T.  \cite{LAN-TEI})
The ranks and border ranks of cubic curves over $\mathbb{C}$ are as in Table 2 (\S 96 in \cite{Seg}).
\end{theorem}
 \begin{table}[htb] \centering
 \begin {tabular}{llllll}
 \textbf{Description}&\textbf{normal form}&\textbf{$rk_{\mathbb{C}}$}&\textbf{$\underline{rk}_{\mathbb{C}}$}&\textbf{Hessian}&\textbf{corr.} \\ 
  & & & & &\textbf{real} \\ 
  & & & & &\textbf{cases} \\ \hline
  1)triple line&$x^3$&1&1& &1 \\ \hline
  2)three\\concurrent lines&$xy(x+y)$&2&2& &2,3\\ \hline  
  3)double line+\\line&$x^2y$&3&2& &4 \\ \hline  
  4)conic+sec. line &$x(x^2+yz)$&4&4&conic+line & 5,6,7 \\ \hline
  5)conic+tangent\\ line&$y(x^2+yz)$&5&3&triple line& 8\\ \hline
  6)irred. Fermat &$y^2z-x^3-z^3$&3&3&triangle &9,10\\ \hline
  7)nodal &$y^2z-x^3-x^2z$&4&4&nodal& 14,15\\ \hline
  8)cusp &$y^2z-x^3$&4&3&2 lines+line &13\\ \hline
  9)smooth for $\lambda \neq -\frac{1}{2},0,1$&$x^3+y^3+z^3+6\lambda xyz$&4&4&smooth & 11\\ \hline  
  10)triangle &$xyz$&4&4&triangle &12,16\\ \hline    
 \end{tabular}
\caption{Ranks and border ranks of plane cubic curves on $\mathbb{C}$}\label{tabrank}
\end{table}

\section{Apolarity}
Let $V$ be a real vector space of dimension n+1 and $V^{\vee}$ its dual space.\\
Following \cite{DOL1} a homogeneous form $\Phi\in S^k(V)$ is 
apolar to a homogeneous form $F\in S^d(V^\vee)$ if $\Phi$ belong
to the kernel of following map:\\
\begin{displaymath}
 ap_F^{k}:S^k(V)\rightarrow S^{d-k}(V^\vee)
\end{displaymath}
given by
\begin{displaymath}
 \Phi\longmapsto P_{\Phi}(F)
\end{displaymath}
where $P_{\Phi}(F)=\Phi\cdot F$ is the contraction operator.\\
This map is called the apolarity map.
The kernel of the apolarity map is the linear space 
(denoted by $AP_{k}$ in \cite{DOL-KAN}) of forms which are apolar to $F$ of degree k.\\
\begin{definition}(cfr. \cite{IAR-KAN})
The matrix Cat(F)  of the above linear map $ap_{F}^{k}$ with respect to two basis of monomials of $S^{k}V$ and 
$S^{d-k}(V^\vee)$ is called the k-th catalecticant matrix of the
homogeneous form $F$ and if $n=2k$ the determinant 
of this matrix is the catalecticant matrix of $F$.\\
\end{definition}
In particular, for  dimV=2, that is for homogeneous polynomials over $\mathbb{P}^1$,after taking an
identification $\wedge^2V=\mathbb{R}$, V can be identified with its dual space $V^{\vee}$ and 
if $f=(ax+by)^d\in S^{d}V$ and $g=(cx+dy)^d\in S^{d}V$ 
the contraction operator of $f$ and $g$ will be $(ad-bc)^d$ up to scalars, and this extends by linearity to every 
$f,g\in S^dV$.\\
Let $f$ be a nonsingular cubic in the projective plane $\mathbb{P}^2(\mathbb{R})$, defined by a homogeneous cubic equation $f(x,y,z)=0$.\\
To find the degenerate polar conic  of $f$, let us write the equation of the polar conic $P_{Y}(f)$ with
respect to the  point Y=$(x_0,y_0,z_0)$ and the basis fixed to identify V and $V^{\vee}$ as:
$$
P_{Y}(f)=x_0\frac{\partial f}{\partial x}+y_0\frac{\partial f}{\partial y}+z_0\frac{\partial f}{\partial z}.
$$

This polar conic splits as a pair of lines if 
\begin{displaymath}
 H(f)=\begin{vmatrix}
    \frac{\partial ^{2}f}{\partial x^2} & 
    \frac{\partial ^{2}f}{\partial xy} & 
    \frac{\partial ^{2}f}{\partial xz} \\
    \frac{\partial ^{2}f}{\partial yx} & 
    \frac{\partial ^{2}f}{\partial y^2} & 
    \frac{\partial ^{2}f}{\partial yz} \\
    \frac{\partial ^{2}f}{\partial zx} & 
    \frac{\partial ^{2}f}{\partial zy} & 
    \frac{\partial ^{2}f}{\partial z^2} 
   \end{vmatrix}=0.
\end{displaymath}
So there are infinitely many points Y such that the polar conics degenerate in two lines and the locus
of such points is given by the above equation that is a cubic curve called the Hessian of $f$.
\begin{definition}
The Hessian curve $H(f)$ of a plane cubic curve $f=0$ is the plane cubic curve defined 
by the equation $H(f)=0$, where $H(f)$ is the determinant
of the matrix of the second partial derivatives of $f$.\\
The inflection points of $f=0$ are the nine points of $\{f=0\}\cap \{H(f)=0\}$.
\end{definition}
The Hessian of a cubic polynomial $F$ is a covariant of $F$(see for example \cite{DOL} pag. 68).

\section{Table of SL(2)-orbits of real binary cubics}
Consider the catalecticant matrix of the cubic binary form 
$f=Ax^3+3Bx^2y+3Cxy^2+Dy^3$ that can be written as 
\begin{equation}
 \left(
\begin{matrix}
 A&B&C\\
 B&C&D\\
\end{matrix}
\right)
\end{equation}
and the discriminant of $f$\\
\begin{displaymath}
 \Delta(f)=\left\bracevert
\begin{array}{cc}
 A&C\\
 B&D\\
\end{array}
\right\bracevert^2-4
\left\bracevert
\begin{array}{cc}
 B&C\\
 C&D\\
\end{array}
\right\bracevert
\left\bracevert
\begin{array}{cc}
A&B\\
B&C\\
\end{array}
\right\bracevert.
\end{displaymath}

Over $\mathbb{R}$, the discriminant $\Delta$ is positive or negative corresponding to $f$ having one real root or
three real distinct roots, respectively.\\
The table over $\mathbb{R}$ is the following (see Table 3)( \cite {COM-OTT}):\\
\begin{table}[htb] \centering
 \begin {tabular}{lllll}
 \textbf{Description}&\textbf{normal form}&\textbf{rk}&\textbf{rk catalec}&{$\Delta$}\\ \hline
  $1)C_{3}$&$x^3$&1&1&  \\ \hline  
  $2)\mathbb{P}^3\setminus Tan(C_{3})_{+}$&$x^3+y^3$&2&2&$>0$\\ \hline
  $3)\mathbb{P}^3\setminus Tan(C_{3})_{-}$&$x(x^2-y^2)$&3&2&$<0$\\ \hline
  $4)Tan(C_{3})\setminus C_{3}$&$x^2y$&3&2&0\\ \hline     
 \end{tabular}
 \caption{Ranks of binary cubics over $\mathbb{R}$} \label{tabrank}
\end{table}
\begin{proposition}
Cases 1,2,3,4 of Table 1 are settled.
\end{proposition}
\begin{proof}
The line 4) of the table 3 means that the normal form $x^2y$ belongs to the tangent of $x^3$ so every point 
on this tangent correspond to a polynomial of the type $x^3+x^2y$ with a double root.

\end{proof}
\section{Real De Paolis algorithm}
There is an algorithm given by De Paolis in XIX century (1886), that gives an useful method to find a decomposition
of a general plane cubic curve as a sum of at most 4 cubes when the first cubic form $l_{0}^3$ is given where $l_{0}$
is a line such that $l_{0}\bigcap H(f)$ is given by three real points.
Let $C_{3}$ a real cubic curve defined by a cubic polynomial $F\in S^{3}\mathbb{R}^3$.\\
It is known that (classic theorem):
\begin{theorem}(\cite{EC}Book 3, Cap.III)
 A nonsingular real cubic has exactly three real inflection points and these points are collinear.\\
\end{theorem}

\begin{proposition}
 Let $F$ be a real ternary cubic and let $l_0$ be a line such that $l_0\cap H(F)$ consists of three distinct points $P_1,P_2,P_3$.
 Then there are defined three real lines $l_i\ni P_i$ (by abuse of notation I will write $P \in l_{i}$ to mean that P belongs
 to the locus defined by $l_{i}=0$) 
 and real scalars $c_i$ (i=0,1,2,3) such that
 $F=\sum_{i=0}^3c_il_i^3$.\\ 
\end{proposition}
\begin{figure}[htbp]\label{depaolis}
 \includegraphics[scale=0.3]{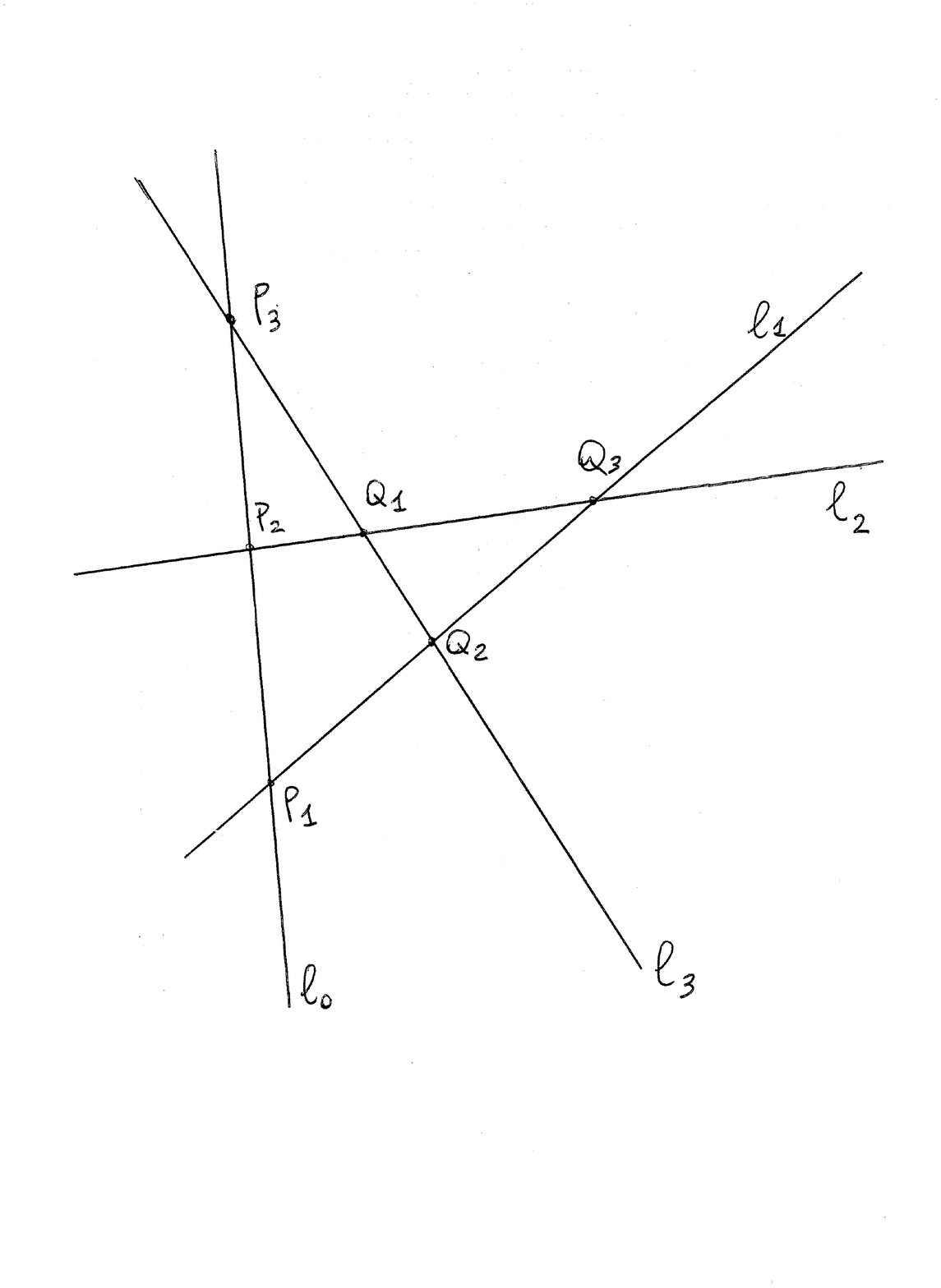}
 \caption{description of De Paolis algorithm}
\end{figure}
The algorithm to find this decomposition is described in the proof.
\begin{proof}
 The singular point of a real singular conic is always real.\\
The algorithm is:\\
INPUT F a real plane cubic and $l_0$ satisfying the assumptions.\\
$l_0$ a line such that $l_0\cap H(F)$ consist of three distinct points.\\
The line $l_0$ joining three real flexes for example (see Figure 1).\\
COMPUTE $l_0\cap H(F) = \left\{ P_1,P_2,P_3 \right\}$.\\
COMPUTE $Q_{i}$ the singular point of the polar conic $P_{P_i}(F)$ for i=1,2,3.\\
COMPUTE $l_1=<P_1,Q_2>$, $l_2=<P_2,Q_1>$, $l_3=<P_3,Q_1>$.\\
SOLVE the linear system $F=\sum_{i=0}^{3} c_{i} l_{i}^{3}$.\\
OUTPUT lines $l_1,l_2,l_3$ and numbers $c_{i}\in \mathbb{R}$,  i=0,1,2,3, such that $F=\sum _{i=0}^{3} c_{i}l_{i}^{3}$ indeed
\begin{displaymath}
P_{P_1}(F)=c_2l_2^2+c_3l_3^3 \quad hence \quad Q_1\in{l_2,l_3}
\end{displaymath}
\begin{displaymath}
P_{P_2}(F)=c_1l_1^2+c_3l_3^3 \quad hence \quad Q_2\in{l_1,l_3}
\end{displaymath}
\begin{displaymath}
P_{P_3}(F)=c_1l_1^2+c_2l_2^3 \quad hence \quad Q_3\in{l_1,l_2}
\end{displaymath}
moreover the points $Q_i$ are real, because singular points of a real singular conic are always real, 
hence the lines $l_i$ are real.\\
This algorithm tells us that for real plane cubics there is only one typical rank which is 4.\\
\end{proof}

\subsection{Hesse pencil}
Every smooth plane cubic is projectively equivalent to a member of the Hesse pencil
\begin{equation}
F_{\lambda }=x^3+y^3+z^3+6\lambda xyz=0.
\end{equation}
The nine base points of the pencil are the flexes of every smooth members of the family.
Three base points are reals, namely $(1,-1,0),(1,0,-1),(0,1,-1)$, moreover there are 
three pairs of conjugate base points.The Hessian of each member of the Hesse pencil is still a member in the 
Hesse pencil.There are four singular members of the Hesse pencil, for
$\lambda=\infty$ (real triangle)
and for $\lambda=-\frac{1}{2}$, $-\frac{\tau}{2}$, $-\frac{\tau^2}{2}$
(imaginary triangle composed by a real line and a pair of complex conjugate lines), where
$\tau$ is a primitive cube root of unity.
Working on real numbers, we will consider just the value $\lambda=-\frac{1}{2}$.\\
\begin{figure} \label{figura1}
 \includegraphics[scale=0.3]{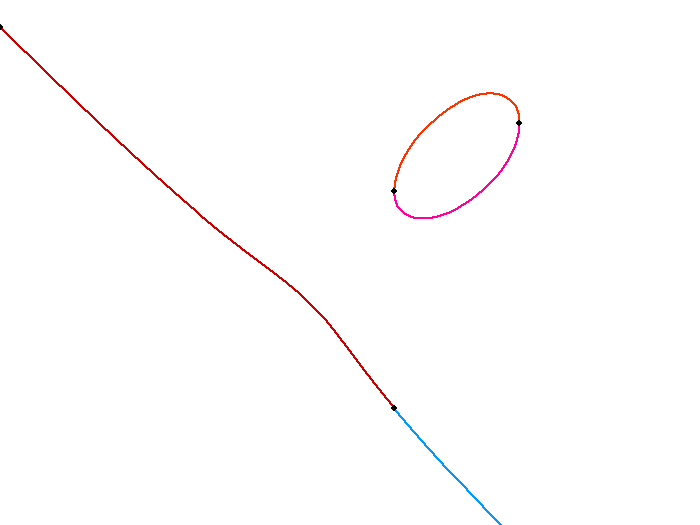}
 \caption{two components $\lambda > -\frac{1}{2}$}
\end{figure}
\begin{figure} \label{figura2}
 \includegraphics[scale=0.3]{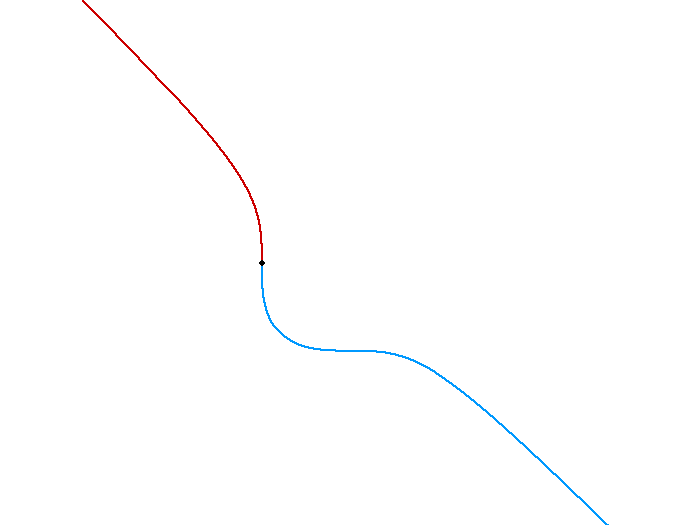}
 \caption{one component $\lambda <-\frac{1}{2}$}
\end{figure}
 \begin{figure} \label{figura3}
  \includegraphics[scale=1]{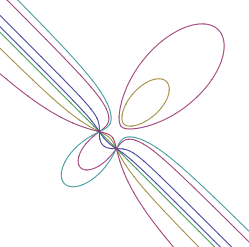}
  \caption{Hesse pencil}
 \end{figure}
We apply De Paolis algorithm to the Hesse pencil (see Figure 4 for the general pencil and Figure 2 and 3 
for the pencil with two components and one component respectively):\\
\begin{equation}
F_{\lambda }=x^3+y^3+z^3+6\lambda xyz=0
\end{equation}
with the condition of non singularity $1+8\lambda^3 \neq 0$. \\ 
$$S=\lambda-\lambda^4=\lambda(1-\lambda)(\lambda^2+\lambda+1)$$ (see \cite{WEB},vol II,ch.12) is, up to scalar,
the invariant of degree four of plane cubics.\\
The other invariant derived from the invariant $S$  is an invariant of the sixth order
in the coefficients, that for the canonical form is 
$$T=1-20\lambda^3-8\lambda^6$$
up to scalar, so for $S=0$ the curve is an equianharmonic cubic
and it is in the orbit of the Fermat cubic, only if $\lambda=0$,or 
$\lambda=1$ on the real numbers.
The discriminant for the Hesse pencil  (4) is (see \cite{SAL}pag. 189 where is denoted by R)
\begin{equation}
\Delta=T^2+64S^3=(1+8\lambda^3)^3.
\end{equation}
We get $\Delta \neq 0$ for $\lambda \neq -\frac{1}{2}$.
\begin{proposition}
Let $F$ a smooth real cubic curve with $S=0$.\\
If $T> 0$ then
$rk_{\mathbb{R}}(F)=3$ and this is the case 9) of Table1 and if $T< 0$ then $rk_{\mathbb{R}}(F)=4$ 
and this is the case 10) of Table 1.
\end{proposition}
\begin{proof}
If $S=0$ over $\mathbb{C}$ then the cubic is a sum of three independent linear powers.\\
There are two cases:
\begin{enumerate}
   \item all the linear forms are real
   \item one is real and two complex conjugate
\end{enumerate}
The second case is given  by the cubic form\\
$$
f=x^3+(y+iz)^3+(y-iz)^3=x^3+2y^3-6yz^2=x^3-(y+z)^3-(y-z)^3+4y^3
$$
with $rk_{\mathbb{R}}(f)=4$ and $rk_{\mathbb{C}}(f)=3$.
\end{proof}

\begin{proposition}
The Waring decomposition of Hesse pencil for $\lambda \neq -\frac{1}{2},0$ is 
$$
F_{\lambda }= c_{0}(x+y+z)^3+c_{1}((1+\lambda)x-\lambda y-\lambda z)^3+c_{2}(-\lambda x+(1+\lambda )y-\lambda z)^3
+c_{3}(-\lambda x-\lambda y+(1+\lambda)z)^3
$$
where $c_{i}$ are described in the proof.
For $\lambda =0$ we have case 9) of Table 1.
\end{proposition}
\begin{proof}
In this proof we apply De Paolis algorithm.\\
The Hessian of (4) is again of this form, that is a curve of the pencil: we have the equation
$$ H(F_\lambda )= -\lambda(x^3+y^3+z^3)+(1+2\lambda^3)xyz=0
$$
and we can choose three real collinear flexes as $P_1=(0,1,-1)$,$P_2=(1,0,-1)$,$P_3=(1,-1,0)$.\\
This flexes belong to the line\\
$$
l_0=x+y+z=0.
$$
Compute the equation of the polar conics $P_{P_i}(F_{\lambda })$ for i=1,2,3:\\
$$P_{P_1}(F_{\lambda })= 3y^2+6\lambda xz-3z^2-6\lambda xy=0$$
$$P_{P_2}(F_{\lambda })= 3x^2+6\lambda yz-3z^2-6\lambda xy=0$$
$$P_{P_3}(F_{\lambda })= 3x^2+6\lambda yz-3y^2-6\lambda xz=0$$
so we get three singular points \\
$$Q_1=(1,\lambda,\lambda)
$$
$$
Q_2=(\lambda,1,\lambda)
$$
$$
Q_3=(\lambda,\lambda,1)
$$
Solving the linear system $$F_{\lambda}=\sum _{i=0}^{3}c_il_i^3$$
we get the value of the coefficients $c_i$.\\
The solution of this system is:
$$c_0=\frac{\lambda(\lambda^2+\lambda+1)}{(2\lambda+1)^2}$$ and
$$c_1=c_2=c_3=\frac{1}{(2\lambda+1)^2}.$$
\end{proof}
In conclusion we settle the cases 9,10,11 of Table 1.
\section{Union of conic and a non-tangent line}
\subsection{Union of imaginary conic and a line}
The cubic has the equation $F=(x^2+y^2+z^2)x$ and its Hessian is 
$$H(F)=(9x^2-y^2-3z^2)(8x).$$
In this case 
$$F_z=2zx=\frac{1}{2}(z+x)^2-\frac{1}{2}(z-x)^2$$
hence
$$F=\frac{1}{6}[(z+x)^3-(z-x)^3]+\phi(x,y)$$
where $\phi(x,y)=x(x^2+y^2)$.\\
A decomposition of $F$ is
$$F=\frac{1}{6}\biggl\{(z+x)^3-(z-x)^3\biggr\}+\frac{1}{2}\biggl\{\frac{1}{3\sqrt{2}}(\sqrt{2}x-y)^3-\frac{1}{3\sqrt{2}}
(-\sqrt{2}x-y)^3\biggr\}$$
so
$$rk_{\mathbb{R}}(F)\leq 4$$
and for \cite{LAN-TEI} $rk_{\mathbb{C}}(F)\geq 4$,then $rk_{\mathbb{R}}(F)=4.$\\
The case 5) of Table 1 is settled.
\subsection{Union  of real conic and a secant line}
In this case the cubic is $F=(x^2+y^2-z^2)y$ and the Hessian is 
$$H(F)=8y(x^2-3y^2-z^2).$$
In this case 
$$F_z=-2zy$$
and like the previous case  we have the decomposition
$$y(y^2+xz)=\frac{1}{96}\biggl\{(4y+x+z)^3+(4y-x-z)^3-2(2y+x-z)^3-2(2y-x+z)^3\biggr\}$$
so
$$rk_{\mathbb{R}}(F)\leq 4$$
and (for \cite{LAN-TEI} or \cite{LAN} ch.10) $rk_{\mathbb{C}}(F)\geq 4$ then $rk_{\mathbb{R}}(F)=4.$\\
The case 7) of Table 1 is settled.
\subsection{Union of a real conic and an external line: the new case of $rk_{\mathbb{R}}= 5$}
In this case the cubic is 
$$
F=(x^2+y^2-z^2)z
$$
and the Hessian of $F$  is 
$$
H(F)=-8z(x^2+y^2+3z^2),
$$
that is the Hessian cubic curve $H(F)=0$ is a imaginary conic plus a line.\\
We can write 
$$
F=x^2z+z(y^2-z^2)=x^2z+\varphi(y,z).
$$
Now 
$$
x^2z=\frac{1}{6}\biggl\{(x+z)^3-(x-z)^3\biggr\}-\frac{1}{3}z^3
$$
so
$$
x^2z+z(y^2-z^2)=\frac{1}{6}\biggl\{(x+z)^3-(x-z)^3\biggr\}-\frac{4}{3}z^3+zy^2
$$
and $\varphi(y,z)=-\frac{4}{3}z^3+zy^2.$\\
But $\varphi(y,z)$ is a binary cubic form with three real roots because
$$zy^2-\frac{4}{3}z^3=z(y-\frac{2}{\sqrt{3}}z)(y+\frac{2}{\sqrt{3}}z)$$
and then (\cite{COM-OTT})
$$rk_{\mathbb{R}}[\varphi(y,z)]=3.$$
Indeed we have
$$
y^2z=\frac{1}{6}\biggl\{(y+z)^3-(y-z)^3-2z^3\biggr\}
$$
so 
$$
zy^2-\frac{4}{3}z^3=\frac{1}{6}\biggl\{(y+z)^3-(y-z)^3\biggr\}-\frac{1}{3}z^3-\frac{4}{3}z^3
$$
that is 
$$
zy^2-\frac{4}{3}z^3=\frac{1}{6}(y+z)^3-\frac{1}{6}(y-z)^3-\frac{5}{3}z^3
$$
and finally the decomposition
$$
F=\frac{1}{6}\biggl\{(x+z)^3-(x-z)^3\biggr\}+\frac{1}{6}(y+z)^3-\frac{1}{6}(y-z)^3-\frac{5}{3}z^3
$$
so 
$$
rk_{\mathbb{R}}(F)\leq 5.
$$
Now we have to prove that the rank of the above cubic can not be smaller than 5.\\
In fact, we have:\\
\begin{theorem}\label{MIO}
The real rank of the reducible cubic given by a real conic plus an external line is 5, that is
$$
rk_{\mathbb{R}}(x^2+y^2-z^2)z=5.
$$
\end{theorem}
\begin{proof}
Suppose
$$
F=(x^2+y^2-z^2)z=l_{1}^3+l_{2}^3+l_{3}^3+l_{4}^3
$$
where $l_{i}$, $i=\{1,2,3,4\}$ are linear real forms.
Let $Q=l_{1}\cap l_{2}$ be the point of intersection of the two lines $l_{1}$ and $l_{2}$. 
Then, up to scalars,the polar conic of F with respect to Q
$$
P_{Q}(F)=l_{3}^2+l_{4}^2
$$
is necessarily singular.\\
Denote $L$ the external line $z=0$.
Then the point $Q\in H(F)$ and $Q\in L=\{z=0\},$ for the particular form of the Hessian, which in this case is 
$$
H(F)=-8z(x^2+y^2+3z^2)
$$
a imaginary conic plus a line.
Moreover, for the same argument, all the intersections $l_{i}\cap l_{j}\in L=\{z=0\}$ so there is only a possibility:\\
the four lines are concurrent in $\tilde Q \in L$ such that $P_{\tilde Q}(F)\equiv 0.$\\
This is impossible because F should be a cone with vertex in Q.\\
The case 6) of Table 1 is settled;the Aronhold invariant $S\neq 0$ because is just so over $\mathbb{C}$ 
then $\underline{rk}_{\mathbb{R}}(x^2+y^2-z^2)z=4.$
\end{proof}
\subsection{Imaginary triangle}
This is the other new case such that the rank over $\mathbb{R}$ is 5.\\
The proof is like the previous one: if the real rank is four we can choose 4 points on the real side 
of the triangle.\\
Let  
\begin{displaymath}
  F=x^3+y^3+z^3-3xyz=(x+y+z)(x^2-xy-xz+y^2-yz+z^2).
\end{displaymath}
Since the Hessian of $F$ coincides with $F$ itself, we may repeat the argument of proof of theorem 6.1
concluding that there are two possibilities:\\
1)\begin{displaymath}
   \lambda F-(x+y+z)^3
  \end{displaymath}
is a Fermat cubic for some $\lambda.$\\
2)the four lines are concurrent in $\tilde Q \in L=(x+y+z)=0$ such that $P_{\tilde Q}(F)\equiv 0.$\\
The first case is excluded because 
the Aronhold invariant $S$ for this pencil is
\begin{displaymath}
 S(\lambda(x^3+y^3+z^3-3xyz)-(x+y+z)^3)=\lambda^4.
\end{displaymath}
Also the second case is excluded because $F$ is not a cone.
Then
\begin{theorem}
 The rank of the imaginary triangle is 5, that is
 \begin{displaymath}
  rk_{\mathbb{R}}(x^3+y^3+z^3-3xyz)=5.
 \end{displaymath}
\end{theorem}
The case 12) is settled.
\subsection{Nodal cubic}
Every plane cubic curve with a real node with two real tangent lines is projectively equivalent to the cubic
$$F=x^3+y^3-3xyz=0.$$
This is a famous cubic curve called ``Folium of Descartes``.\\
It has a double point in $(0,0,1)$ and there has a node with tangents $x=0$ and $y=0.$
The Hessian is $(x^3+y^3+xyz)(-54).$
In this case 
$$F_x=3x^2-3yz$$
$$F_y=3y^2-3xz$$
$$F_z=-3xy$$
Let 
$$\alpha F_x+\beta F_y+\gamma F_z=0$$
be the equation of the polar conic with respect to the point $(\alpha,\beta,\gamma).$
This is a reducible conic if
$$
A= \left(
\begin{array}{ccc}
 3\alpha & \frac{-\gamma}{2} & \frac{-\beta}{2}\\
 -\frac{\gamma}{2} & \beta &-\frac{\alpha}{2}\\
-\frac{\beta}{2} & -\frac{\alpha}{2} & 0
\end{array}
 \right)
$$
We deduce that $\det A=0$ if $(\alpha,\beta,\gamma)=(1,-1,0).$
So the pencil is 
$$3x^2-3yz-3y^2+3xz$$
that factors as
$$3(x+y+z)(x-y).$$
So we can write the polar of F at the point $(1,-1,0)$ as:\\
$$F_x-F_y=3\biggl\{(x+y+z)(x-y)\biggr\}=3\biggl\{\frac{1}{4}(2x+z)^2-\frac{1}{4}(2y+z)^2\biggr\}$$
because of the identity
$$ab=\frac{(a+b)^2}{4}-\frac{(a-b)^2}{4}.$$
So we have, integrating with respect to $x$ and $y$,
$$F=\frac{1}{8}\biggl\{(2x+z)^3-(2y-z)^3\biggr\}+\biggl\{function\ of\ two\ variables\biggr\}.$$
To find this function let us write the equality
$$x^3+y^3-3xyz=\frac{1}{8}\biggl\{(2x+z)^3-(2y+z)^3\biggr\}-z\biggl\{\frac{3}{2}x^2+\frac{3}{4}xz+\frac{1}{4}z^2+
\frac{3}{2}y^2+\frac{3}{4}yz+3xy\biggr\}.$$
Let $$g(x,y,z)=\frac{3}{2}x^2+\frac{3}{4}xz+\frac{1}{4}z^2+
\frac{3}{2}y^2+\frac{3}{4}yz+3xy$$ be this cubic form; it depends on two essential variables, namely $z$ and $h=x+y$.\\
We have
$$g(z,x+y)=g(z,h)=z\biggl\{\frac{z^2}{4}+\frac{3}{4}zh+\frac{3}{2}h^2\biggr\}$$
and the discriminant of the polynomial of degree 2 into the square bracket
$$z^2+3zh+6h^2$$
is 
$$\Delta=9-6\cdot 4\textless0$$
so this quadratic polynomial has rank 2  and the cubic nodal form has rank $\leq 4$.\\
Then
$$
g(z,h)=2\biggl\{\frac{1-\sqrt{5}}{4}z+h\biggr\}^3-2\biggl\{\frac{1+\sqrt{5}}{4}z-h\biggr\}^3
$$
and
$$
x^3+y^3-xyz=\frac{1}{8}\biggl\{(2x+z)^3+(2y+z)^3\biggr\}+2\biggl\{\frac{1-\sqrt{5}}{4}z+(x+y)\biggr\}^3-
2\biggl\{\frac{1+\sqrt{5}}{4}z-(x+y)\biggr\}^3.
$$
Then
$$
rk_{\mathbb{R}}(F)\leq 4
$$
and again for \cite{LAN-TEI} we conclude that $rk_{\mathbb{R}}=4$ because 
$$
4\geq rk_{\mathbb{R}}(F)\geq rk_{\mathbb{C}}(F)=4.
$$
The case 14) of Table 1 is settled.
\subsection{Cubica punctata}
Let 
$$f=y^2z-x^3+x^2z$$
be the normal form of the so called ``cubica punctata'', that is any irreducible cubic with a double point in the origin
having two complex
tangent lines $x^2+y^2=0=(x+iy)(x-iy)=0$.\\
We have
$$
f_{x}=x(2z-3x)=\frac{1}{4}\biggl\{(-2x+2z)^2-(4x-2z)^2\biggr\}
$$
so 
$$
f=\frac{1}{4}\left(-\frac{1}{2}\frac{(-2x+2z)^3}{3}\right)-\frac{1}{4}\left(\frac{(4x-2z)^3}{3}\right)+\phi(y,z)
$$
with 
$$
\phi(y,z)=y^2z+\frac{1}{6}z^3=z(y^2+\frac{1}{6}z^2)
$$
so $\phi$ is a binary cubic form with only one real root and for \cite{COM-OTT}  has rank 2.\\
We get
$$
\phi(y,z)=\frac{1}{12}\left(\sqrt{2}y+z\right)^3-\frac{1}{12}\left(\sqrt{2}y-z\right)^3
$$
so
$$
f=\frac{1}{4}\biggl\{-\frac{1}{2}\frac{(-2x+2z)^3}{3}-\frac{1}{4}\frac{(4x-2z)^3}{3}\biggr\}+
  \frac{1}{12}\left(\sqrt{2}y+z\right)^3-\frac{1}{12}\left(\sqrt{2}y-z\right)^3
$$
and finally
$$
rk_{\mathbb{R}}(f)\leq 4.
$$
Again 
$$
4\geq rk_{\mathbb{R}}(f)\geq rk_{\mathbb{C}}(f)=4
$$
so the case 15) of Table 1 is settled.
\subsection{Union of real conic plus tangent line}
Let's see the case of the irreducible conic plus tangent line where the rank is five.\\
$$
 F=18y(x^2+yz)
$$
We have the decomposition
$$
18y(x^2+yz)=(3y-x)^3+(3y+x)^3+3(z-2y)^3+3(z-4y)^3-6(z-3y)^3.
$$
We get
\begin{displaymath}
 rk_{\mathbb{R}}(F)\leq 5
\end{displaymath}
and again $rk_{\mathbb{R}}(F)\geq rk_{\mathbb{C}}(F)\geq 5$.\\
The case 8) of Table 1 is settled.\\

The case 13) is also settled in \cite{LAN-TEI} where the decomposition is given.\\
 We have
\begin{displaymath}
y^2z-x^3=\frac{1}{6}\biggl\{(y+z)^3-(y-z)^3-2z^3\biggr\}-x^3
\end{displaymath}
The last case of the triangle is settled in \cite{LAN-TEI}. 
The decomposition is :
\begin{displaymath}
xyz=\frac{1}{24}\biggl\{(x+y+z)^3-(-x+y+z)^3-(x-y+z)^3-(x+y-z)^3\biggr\}.
\end{displaymath}
The case 16) is  settled.\\
This concludes the proof of Theorem \ref{Maurizio-Giorgio}.

\subsection*{Acknowledgments}
The author would like to warmly thank Giorgio Ottaviani for suggesting the problem and for his help and support and also thank
the anonymous referee for useful comments and remarks.


\begin{thebibliography}{10}

\bibitem{BO}
A.Bernardi, G. Blekherman, and G. Ottaviani, {\em On real typical ranks}, arXiv:1512.01853.
\bibitem{BLE}
Blekherman, G., {\em Typical real ranks of binary forms}, arXiv:1205.3257v1 to appear on Foundations of Computational Math.,2012.
\bibitem{COM-SEI}
G. Comas and M. Seiguer,{\em On the ranks of a bynary forms}, Found. Comput. Math. 11 (2010),65-78.
\bibitem{COM-MOU}
P. Comon and B. Mourrain, {\em Decomposition of quantics in sums of powers of linear forms}, Signal Processing 53 (1996), 93-107.
\bibitem{COM-OTT}
P. Comon and G. Ottaviani, {\em On the typical Ranks of Real Bianry Forms}, Linear and Multilinear Algebra, 60 (2012), no. 6, 657-667.
\bibitem{DOL1}
I. Dolgachev, {\em Classical Algebraic Geometry},Cambridge, 2012.
\bibitem{DOL}
I. Dolgachev, {\em Lectures on Invariant Theory },Cambridge, London Math. Soc.,vol. 296,(2003).
\bibitem{DOL-KAN}
I.Dolgachev and V. Kanev, {\em Polar covariant of plane cubics and quartics}, Advances in Math. 98 (1993), 216-301.
\bibitem{EC}
F. Enriques and O. Chisini, {\em Lezioni sulla Teoria Geometrica delle Equazioni e delle Funzioni Algebriche}, Zanichelli, Bologna,1985.
\bibitem{IAR-KAN}
A. Iarrobino and V. Kanev,{\em Power Sums, Gorestein Algebras and Determinantal Loci}, Lectures Notes in Math., vol. 1721, Springer,1999.
\bibitem{LAN-TEI}
J. M. Landsberg and Z. Teitler, {\em On the ranks and Border Ranks of Symmetric Tensors}, Found.Comput. Math. 10 (2010),no.3,339-366.
\bibitem{LAN}
J.M. Landsberg, {\em Tensors:Geometry and Applications}, vol.128, AMS, 2012, Graduate Studies in Mathematics.
\bibitem{SAL}
G. Salmon, {\em Modern Higher Algebra}, Hodges, Figgis and Co., 1885.
\bibitem{Seg}
B. Segre, {\em The non singular cubic surfaces}, Oxford,1942.
\bibitem{WEB}
H. Weber {\em Lehrbuch der Algebra}, vol.2, F. vieweg snd Son, Braunschweig, 1898.



\end{thebibliography}
\end{document}